\documentclass[12pt]{amsart}
\usepackage[normalem]{ulem}
\usepackage{lineno}
\usepackage{soul}
\usepackage{graphicx}
\usepackage{xcolor}
\usepackage{amssymb, amsmath, amsthm, hyperref, euscript, color}
\usepackage[dvipsnames]{xcolor}
\usepackage{amscd}
\usepackage{tikz-cd}
\usepackage{bbm}
\usepackage{enumitem}
\usepackage[english]{babel}
\usepackage{mathtools}
\usepackage{makecell}
\usepackage{multirow}

\numberwithin{equation}{section}

\newtheoremstyle{theor}{6pt plus 1pt minus 1pt}{6pt plus 1pt minus 1pt}{\slshape}{}{\bfseries}{.}{5pt plus 1pt minus 1pt}{}
\newtheoremstyle{def}{6pt plus 1pt minus 1pt}{6pt plus 1pt minus 1pt}{}{}{\bfseries}{.}{5pt plus 1pt minus 1pt}{}
\newtheoremstyle{rmk}{6pt plus 1pt minus 1pt}{6pt plus 1pt minus 1pt}{}{}{\bfseries}{.}{5pt plus 1pt minus 1pt}{}
\newtheoremstyle{claim}{6pt plus 1pt minus 1pt}{6pt plus 1pt minus 1pt}{\slshape}{}{\bfseries}{.}{5pt plus 1pt minus 1pt}{}

\theoremstyle{theor}
\newtheorem{newstatement}{newstatement}[section]
\newtheorem{lemma}[newstatement]{Lemma}
\newtheorem{theorem}[newstatement]{Theorem}

\newtheorem{proposition}[newstatement]{Proposition}

\theoremstyle{def}
\newtheorem{definition}[newstatement]{Definition}

\theoremstyle{rmk}
\newtheorem{remark}[newstatement]{Remark}

\newtheorem*{example*}{Example}
\newtheorem*{que*}{Question}
\newtheorem*{rmk*}{Remark}

\theoremstyle{claim}

\theoremstyle{theor}
\newtheorem{thm}{Theorem}

\newtheorem{cor}[thm]{Corollary}

\expandafter\let\expandafter\oldproof\csname\string\proof\endcsname
\let\oldendproof\endproof
\renewenvironment{proof}[1][\proofname]{%
  \oldproof[\slshape #1]%
}{\oldendproof}
\def\provedboxcontents#1{$\square$}

\makeatletter
\newsavebox\myboxA
\newsavebox\myboxB
\newlength\mylenA

\newcommand*\xoverline[2][0.75]{%
    \sbox{\myboxA}{$\m@th#2$}%
    \setbox\myboxB\null
    \ht\myboxB=\ht\myboxA%
    \dp\myboxB=\dp\myboxA%
    \wd\myboxB=#1\wd\myboxA
    \sbox\myboxB{$\m@th\overline{\copy\myboxB}$}
    \setlength\mylenA{\the\wd\myboxA}
    \addtolength\mylenA{-\the\wd\myboxB}%
    \ifdim\wd\myboxB<\wd\myboxA%
       \rlap{\hskip 0.5\mylenA\usebox\myboxB}{\usebox\myboxA}%
    \else
        \hskip -0.5\mylenA\rlap{\usebox\myboxA}{\hskip 0.5\mylenA\usebox\myboxB}%
    \fi}
\makeatother

\makeatletter

\def\pput{\@ifnextchar[{\@put}{\@@rput[\z@,\z@][r]}}
\def\@put[#1]{\@ifnextchar[{\@@put[#1]}{\@@@@@put[#1]}}
\def\@@put[#1][{\@ifnextchar{l}{\@@lput[#1][}{\@@@put[#1][}}
\def\@@@put[#1][{\@ifnextchar{c}{\@@cput[#1][}{\@@@@put[#1][}}
\def\@@@@put[#1][{\@ifnextchar{r}{\@@rput[#1][}{\relax}}
\def\@@@@@put[{\@ifnextchar{l}{\@@lput[\z@,\z@][}{\@@@@@@put[}}
\def\@@@@@@put[{\@ifnextchar{c}{\@@cput[\z@,\z@][}{\@@@@@@@put[}}
\def\@@@@@@@put[{\@ifnextchar{r}{\@@rput[\z@,\z@][}{\@@@@@@@@put[}}
\def\@@@@@@@@put[#1]{\@@rput[#1][r]}

\let\hm@d@\leavevmode

\long\def\@@lput[#1,#2][l]#3{\setbox0\hbox{#3}\hm@d@\raise#2\hbox to\z@{\dimen0 #1%
  \advance\dimen0-\wd0\kern\dimen0\dp0\z@\ht0\z@\wd0\z@\box0\hss}\ignorespaces}
\long\def\@@cput[#1,#2][c]#3{\setbox0\hbox{#3}\hm@d@\raise#2\hbox to\z@{\dimen0 #1%
  \advance\dimen0-.5\wd0\kern\dimen0\dp0\z@\ht0\z@\wd0\z@\box0\hss}\ignorespaces}
\long\def\@@rput[#1,#2][r]#3{\setbox0\hbox{\kern#1\raise#2\hbox{#3}}%
  \dp0\z@\ht0\z@\wd0\z@\hm@d@\box0\ignorespaces}

\def\widebar{\ifx\math@version\@@bold
  \let\@widebar\@@@widebar\else\let\@widebar\@@widebar\fi\@widebar}

\def\@@widebar#1{\hbox{\setbox15\hbox{$#1$}%
  \dimen15 0.45\wd15\advance\dimen15 0.15\ht15%
  \dimen16\ht15\advance\dimen16 0.00em\advance\dimen16 0.3ex%
  \dimen17 0.65\wd15\advance\dimen17 0.05\ht15\advance\dimen17 0.1ex%
  \dimen18 0.035em\advance\dimen18 0.00ex
  \pput[\dimen15,\dimen16][c]{\vrule depth 0pt height \dimen18 width \dimen17}}#1}

\def\@@@widebar#1{\hbox{\setbox15\hbox{$#1$}%
  \dimen15 0.45\wd15\advance\dimen15 0.15\ht15%
  \dimen16\ht15\advance\dimen16 0.00em\advance\dimen16 0.26ex%
  \dimen17 0.65\wd15\advance\dimen17 0.05\ht15\advance\dimen17 0.1ex%
  \dimen18 0.05em\advance\dimen18 0.00ex
  \pput[\dimen15,\dimen16][c]{\vrule depth 0pt height \dimen18 width \dimen17}}#1}

\makeatother

\newcommand{\Z}{\mathbb{Z}}

\newcommand{\K}{\mathrm{K}3}

\newcommand{\defeq}{\vcentcolon=}

\DeclareMathOperator{\Int}{Int}

\newcommand{\cs}{\mathbin{\#}}

\newcommand{\CP}{{\mathbb C\mkern-0.5mu\mathrm P}}
\newcommand{\CPbar}{\widebar{{\mathbb C\mkern-0.5mu\mathrm P}}}

\renewcommand{\:}{\,{:}\;}
\DeclareMathOperator{\Tor}{Tor}

\DeclareMathOperator{\GL}{GL}

\DeclareMathOperator{\lk}{lk}

\let\geq\geqslant
\let\leq\leqslant

\newcommand{\trans}[2]
  {({\text{\Large$\scriptstyle#1$}}\mkern8mu{\text{\Large$\scriptstyle#2$}})}

\newcommand{\darrow}[1]
{\mathchoice
{\xlongrightarrow{\hbox{\raisebox{0pt}[0pt][0pt]
  {\smash{$\scriptstyle\mkern1mu#1
  \mkern3mu:\mkern2mu1\mkern1.5mu$}}}}}
{\xlongrightarrow{\hbox{\raisebox{-1.5pt}[0pt][0pt]
  {\smash{$\scriptstyle\mkern2mu#1
  \mkern3mu:\mkern2mu1\mkern2.5mu$}}}}}
{}{}}

\tikzcdset{every label/.append style = {font=\normalsize,inner sep=0.75ex}}
\tikzcdset{every matrix/.append style = {row sep=5ex,column sep=1em}}

\newcommand{\Case}[1]{\par\noindent\textsl{Case #1}}
\newcommand{\Cases}[1]{\par\noindent\textsl{Cases #1}}

\begin{document}

\title[Branched coverings of simply connected $4$-manifolds]{Branched coverings\\of simply connected $4$-manifolds}

\author{Valentina Bais}
\address{Scuola Internazionale Superiore di Studi Avanzati (SISSA), Via Bo\-nomea 265\\34136\\Trieste\\Italy.}
\email{vbais@sissa.it}

\author{Riccardo Piergallini}
\address{Scuola di Scienze e Tecnologie,
Università di Camerino\\Via Madon\-na delle Carceri 9\\62032 Camerino\\Italy.}
\email{riccardo.piergallini@unicam.it}

\author{Daniele Zuddas}
\address{Dipartimento di Matematica, Informatica e Geoscienze\\Università degli Studi di Trieste \\Via Valerio 12/1\\34127\\Trieste\\Italy.}
\email{dzuddas@units.it}



\begin{abstract}
We show that, given $d \geq 4$ and two closed connected oriented PL $4$-manifolds $M$ and $N$ such that $N$ has a handle decomposition with no $1$- and $3$-handles, there exists a $d$-fold (simple) branched covering $p \colon M \darrow{d} N$ if and only if there is an isometric embedding of lattices $d \cdot I_N \hookrightarrow I_M$. Here $I_N$ and $I_M$ respectively denote the intersection lattices of $N$ and $M$. In particular, we characterize the manifolds which are branched covers of the $\K$ surface.
\end{abstract}

\maketitle

\kern-12pt

\section{Introduction}

By a classical result of Alexander \cite{A1920}, any closed connected oriented PL $n$-manifold $M$ admits a branched covering
\[p \colon M \darrow{d} S^n\]
of some degree $d \geq 1$. One possible interpretation of this fact is that any such manifold $M$ can be presented by means of a codimension two closed subpolyhedron $B_p \subset S^n$, i.e. the branch set, and a group homomorphism $\omega_p \colon \pi_1(S^n \setminus B_p) \rightarrow S_d$, i.e. the monodromy representation of the unbranched part of $p$, where $S_d$ denotes the permutation group of $d$ elements. We remark that such a presentation is not optimal, since the degree $d$ depends on the cardinality of a chosen triangulation of $M$ and the branch set $B_p$ is not a PL submanifold. However, the following improvements of Alexander's result in low dimensions are known. The easiest case is when $M = \Sigma_g$ is a closed oriented surface of genus $g$: one gets a simple $d$-fold branched covering $\Sigma_g \darrow{d} S^2$ for any degree $d \geq 2$ by stabilizing the quotient map induced by the hyperelliptic involution. Moving to the case $n=3$, Hilden \cite{Hi}, Hirsch \cite{Hirsch} and Montesinos \cite{M} independently proved that any closed oriented $3$-manifold can be represented by means of a $3$-fold dihedral covering of $S^3$ branched over a knot. The $4$-dimensional case has been considered by Piergallini and Iori \cite{P,PI}, who showed that any closed connected oriented PL $4$-manifold $M$ can be represented by a simple $d$-fold branched covering $p\colon M \darrow{d} S^4$ for any $d \geq 4$, where the branch set $B_p \subset S^4$ is a locally flat PL embedded surface if $d \geq 5$ while it is a locally flat PL self-transversally immersed surface if $d=4$. The case when $d=3$ and $M$ has no 1-handles has been studied in \cite{BCKM24} using trisections.

In general, one would hope to give a complete answer to the following question.

\begin{que*}
    Given two PL $n$-manifolds $M$ and $N$ and a natural number $d \geq 2$, is there a $d$-fold branched covering
    \[M \darrow{d} N?\]
\end{que*}

Criteria for the existence of a simple $d$-fold branched covering\break $p \colon M \darrow{d} N$ with $N = \cs_{k_1} \CP^2 \cs_{k_2}\! \CPbar^2$ or $N = \cs_{k} (S^2 \times S^2)$ are given by Piergallini and Zuddas \cite{PiergalliniZuddas1} in terms of the Betti numbers $b_2^\pm(M)$. Their proof heavily relies on the fact that the $4$-manifold $N$ is the union of a plumbing of disk bundles over $2$-spheres with a single $4$-handle. On the other hand, by a result of Neumann and Weintraub \cite{Neumann}, the connected sums $N$ as above are the only $4$-manifolds which can be obtained by such a plumbing. As a consequence, the approach followed in \cite{PiergalliniZuddas1} cannot be extended to other choices of~$N$.

We started this work by thinking about the branched coverings of the $\K$ surface, which together with the basic manifolds $\CP^2$, $\CPbar^2$ and $S^2 \times S^2$, are the main building blocks needed to construct all the possible homeomorphism types of closed simply-connected PL $4$-mani\-folds by taking connected sums.

Then we came up with new ideas that allowed us to generalize the above mentioned results of \cite{PiergalliniZuddas1} to all the closed 4-manifolds $N$ which admit a handle decomposition without 1- and 3-handles, thus providing a partial answer to the above Question for simply connected PL $4$-manifolds. We recall that it is a long standing open problem in low dimensional topology that any closed PL simply connected $4$-manifold $N$ admits a handle decomposition with no $1$- and $3$-handles, see \cite[Problem 4.18]{KirbyList} in  Kirby's list.

\begin{thm}\label{thmA}
   Let $M$ and $N$ be two closed connected oriented PL $4$-manifolds such that $N$ admits a handle decomposition with no $1$- and $3$-handles. For any $d \geq 4$, there exists a $d$-fold branched covering \[p\colon M \darrow{d} N\] if and only if there exists an isometric embedding of lattices
   \[d \cdot I_N \hookrightarrow I_M,\]
   where $I_N$ and $I_M$ respectively denote the intersection lattices of $N$ and $M$. Moreover, if this is the case, one can build $p$ in such a way that it is simple and its branch set $B_p \subset N$ is a locally flat PL embedded surface if $d \geq 5$ while it is a locally flat PL self-transversally immersed surface if $d=4$.
\end{thm}
In Proposition \ref{prop:inequalities}, we show that there exists an isometric embedding $d \cdot I_N \hookrightarrow I_M$ for some integer $d \geq 1$ if and only if $b_2^{\pm}(N) \leq b_2^{\pm}(M)$, and we provide a list of achievable values of $d$, depending on the parity and the signature of the two intersection forms $\beta_N$ and $\beta_M$. This provides an alternative proof of \cite[Theorem 1.3, (a) and (b)]{PiergalliniZuddas1}.

\medskip

The next corollary follows from Theorem \ref{thmA} combined with a theorem of Hopf \cite{Hopf}, see also \cite[Proposition 2.6]{Harpe}. In particular, it provides a criterion to check the existence of a dominant map $f \colon M \rightarrow N$ when $N$ is a closed oriented smooth $4$-manifold admitting a handle decomposition with no $1$- and $3$-handles, cf. \cite[Question 2.2]{Harpe}. 
\begin{cor}\label{cor:b}
    Let $M$ and $N$ be as in Theorem \ref{thmA}. Then there exists a branched covering $p \colon M \rightarrow N$ if and only if there is a continuous dominant map $f \colon M \rightarrow N$ of positive degree. Moreover, if the degree of $f$ is $\deg f \geq 4$, $f$ is homotopic to a branched covering.
\end{cor}

\begin{rmk*}
   Corollary \ref{cor:b} provides a new proof of a particular case in \cite[Theorem 3]{DW2003}. More precisely, Duan and Wang show that, given two closed connected oriented $4$-manifolds $M$ and $N$ with $N$ simply connected, there exists a dominant map $f\colon M \rightarrow N$ of degree $d$ if and only if there is an isometric embedding of lattices $d \cdot I_N \hookrightarrow I_M$. The two results coincide when either $I_N$ is odd or $I_N$ is even and $b_2(M) \geq 11/8  \,|\sigma(M)|$, since in both cases $N$ is homeomorphic to a $4$-manifold with no $1$-and $3$-handles. This follows from Freedman's classification of simply connected topological $4$-manifolds \cite{Freedman}, Donaldson's diagonalization theorem of simply connected definite smooth $4$-manifolds \cite{Do83} and Serre's classification of integral unimodular bilinear forms \cite{Se73,MH73}.
\end{rmk*}

As an application of Theorem \ref{thmA} and Proposition \ref{prop:inequalities}, we have the following result.
\begin{cor}\label{cor:c}
    Let $M$ be a closed connected oriented PL $4$-manifold. Then there exists a $d$-fold branched covering
    \[p \colon M \darrow{d} \K\]
    for a suitable integer $d \geq 1$ if and only if $b_2^+(M) \geq 3$ and $b_2^-(M)\geq 19$. Moreover, if this is the case, the covering $p$ can be assumed to be simple and its degree can be any of the following: any even $d \geq 4$ if $\beta_M$ is odd; any integer $d\geq 4$ if $\beta_M$ is even. Finally, if this is the case, one can build $p$ in such a way that it is simple and its branch set $B_p \subset \K$ is a locally flat PL embedded surface if $d \geq 5$ while it is a locally flat PL self-transversally immersed surface if $d=4$.
\end{cor}
The next corollary is an immediate consequence of Theorem \ref{thmA} and Proposition \ref{prop:inequalities}.

\begin{cor}\label{cor:d}
    Let $N$ be a closed connected oriented PL $4$-manifold admitting a handle decomposition with no $1$- and $3$-handles and let $N'$ be any smooth $4$-manifold homeomorphic but not necessarily diffeomorphic to $N$. Then there exists a $d$-fold branched covering
    \[p \colon N' \darrow{d} N\]
    for a suitable $d \geq 4$, according to Proposition \ref{prop:inequalities}. In particular, the value $d=4$ is always allowed. Moreover, if this is the case, one can build $p$ in such a way that it is simple and its branch set $B_p \subset N$ is a locally flat PL embedded surface if $d \geq 5$ while it is a locally flat PL self-transversally immersed surface if $d=4$.
\end{cor}

The paper is organized as follows: the next three sections contain basic definitions and technical preliminary results; Theorem \ref{thmA} is proved in Section \ref{proofs/sec}; finally, Corollaries \ref{cor:b} and \ref{cor:c} are proved in Section \ref{proofcorBC}. The proof of Corollary \ref{cor:d} is trivial and is hence omitted.

\subsection*{Acknowledgments}

The authors acknowledge GNSAGA – Istituto Nazionale di Alta Matematica ``Francesco Se\-ve\-ri'', Italy.

\subsection*{Notation and conventions}
We work in the category of PL manifolds. Notice that this coincides with the category of smooth manifolds for dimension $\leq 4$.

If $M$ is an oriented manifold, we indicate by $\widebar M$ the same manifold with reversed orientation.

Given a locally flat PL immersed submanifold $S \subset N$, we denote by $\nu(S) \subset N$ a regular neighborhood of $S$ in $N$. This is diffeomorphic to the total space of the normal disk bundle of $S$ when $S$ is locally flat PL embedded and to a self-plumbing of such bundle when $S$ is locally flat PL self-transversally immersed and $\dim (S)=\dim (N) /2$.

\newpage

\section{Branched coverings}\label{BC/sec}

We start this section by recalling the definition of branched covering in the PL category.

\begin{definition}
    Given two compact PL $n$-manifolds $M$ and $N$, a PL map
    \[p\colon M \darrow{d} N\]
    is called a \textsl{$d$-fold branched covering} if it is non-degenerate and it restricts to a $d$-fold ordinary covering over the complement of a codimension $2$ closed subpolyhedron of $N$. The \textsl{branch set} of $p$ is the smallest subpolyhedron $B_p \subset N$ with this property, while the \textsl{degree} $d=d(p)$ of $p$ is the cardinality of the preimage of any point in $N \setminus B_p$.

    Moreover, $p$ is called a \textsl{simple} branched covering, if the monodromy of any meridian of the branch set $B_p$ is a transposition.

    A $d$-fold branched covering as above is fully determined up to PL homeomorphisms by the pair $(N,B_p)$ and the \textsl{monodromy} representation 
    \[\omega_p\colon \pi_1(N \setminus B_p) \longrightarrow S_d\] 
    associated to the ordinary covering $p|\colon M \setminus p^{-1}(B_p) \darrow{d} N \setminus B_p$, where $S_d$ denotes the permutation group of $\{1,\dots,d\}$ (see \cite{Fox}). 
\end{definition}

\begin{remark} 
The monodromy of a given branched covering $p$ is usually described by means of its values on a given generator system for $\pi_1(N \setminus B_p)$. This is usually chosen to be a Hurwitz system when $N=S^2$ and $B_p$ is a finite set of points. If  $N=S^3$ and $B_p$ is a link, we fix a projection diagram and consider a Wirtinger system of meridians. Attaching to any such generator of $\pi_1(S^n \setminus B_p)$ its image through the monodromy $\omega_p$ is called a \textsl{labeling} or a \textsl{coloring} of $B_p$.
\end{remark}

\section{Intersection forms}

Given a smooth closed oriented $4$-manifold $M$, we denote by \[\beta_M \colon H_2(M;\Z)/\Tor H_2(M;\Z)\times H_2(M;\Z)/\Tor H_2(M;\Z) \longrightarrow \Z\] its intersection form on second homology modulo torsion and by \[I_M \defeq (H_2(M;\Z)/\Tor H_2(M;\Z), \beta_M)\] the integral unimodular inner product intersection space of $M$. If $k$ is a non-zero integer, we set \[k \cdot I_M \defeq (H_2(M;\Z)/\Tor H_2(M;\Z), k \cdot \beta_M).\]

In particular, we set \[\langle k \rangle = k \cdot \langle 1 \rangle\] where $\langle 1 \rangle$ is the rank one positive definite unimodular lattice.
Moreover, we let $H$ and $E_8$ respectively denote the unimodular lattices represented by the homonymous matrices
\[H = \left(\begin{matrix}\, 0 &1\\ \,1 &0\,\end{matrix} \right)
   \quad \text{and} \quad
   E_8 = \left(\begin{matrix}
    \,2 & 1 & 0 & 0 & 0 & 0 & 0 & 0\\
    \,1 & 2 & 1 & 0 & 0 & 0 & 0 & 0\\
    \,0 & 1 & 2 & 1 & 0 & 0 & 0 & 0\\
    \,0 & 0 & 1 & 2 & 1 & 0 & 0 & 0\\
    \,0 & 0 & 0 & 1 & 2 & 1 & 0 & 1\\
    \,0 & 0 & 0 & 0 & 1 & 2 & 1 & 0\\
    \,0 & 0 & 0 & 0 & 0 & 1 & 2 & 0\\
    \,0 & 0 & 0 & 0 & 1 & 0 & 0 & 2\, 
    \end{matrix}\,\right).\]

We start by proving the following lemma, which immediately gives one of the two implications of Theorem \ref{thmA}.

\begin{lemma}\label{lem:necessary}
    Let $p\colon M \darrow{d} N$ be a $d$-fold branched covering between two closed connected oriented PL $4$-manifolds. Then $d\cdot I_N$ admits an isometric embedding in $I_M$.
\end{lemma}

\begin{proof}
    Let $S_1, \dots, S_k \subset N$ be PL embedded surfaces representing a basis of the free abelian group $H_2(N;\Z)/ \Tor H_2(N;\Z) \cong \Z^{k}$, where $k=b_2(N)$. Without loss of generality, we can assume that all such surfaces are transverse to the branch set $B_p \subset N$ and we set $\widetilde{S}_i = p^{-1}(S_i)$ for $i=1, \dots, k$. Then $\widetilde S_1, \dots, \widetilde S_k \subset M$ are PL embedded surfaces whose homology classes satisfy \[\beta_M([\widetilde S_i], [\widetilde S_j]) = d \cdot \beta_N([S_i],[S_j])\]for all $i,j=1, \dots, k$. The conclusion follows.
\end{proof}

The next proposition gives a readily verifiable criterion for understanding whether an embedding of lattices as in Theorem \ref{thmA} exists or not, and it provides a non-exhaustive list of values for the degree $d$.

\begin{proposition}\label{prop:inequalities}
    Let $M$ and $N$ be closed connected oriented PL $4$-manifolds. Then there exists a positive integer $d$ such that $d\cdot I_N$ admits an isometric embedding in $I_M$ if and only if $b_2^+(N)\leq b_2^+(M)$ and $b_2^-(N)\leq b_2^-(M)$. Moreover, if these inequalities are satisfied, we can choose $d$ as follows, for every $k\geq 1$.
\begin{center}
\vspace{2mm}
\begin{tabular}{|c|c|c|c|l|}
\hline
\vrule width 0pt height 11pt
\text{Case} & $d$ & $\beta_N$ & $\beta_M$ & \text{extra conditions}\\
\hline & & & & \\[-13pt] \hline
\vrule width 0pt height 11pt
1 & $1$ &  \multirow{4}{2em}{\hfill odd\hfill} & 
\multirow{2}{2em}{\hfill odd\hfill}  & \\ 
\cline{1-2}\cline{5-5} 
\vrule width 0pt height 11pt
2 &  $5$ &   &  &  $b_2(N) \leq b_2(M)/2$\\ 
\cline{1-2}\cline{4-5}
\vrule width 0pt height 11pt
3 &$2k$ &   & \multirow{2}{2em}{\hfill even\hfill}  & $\sigma(M)=0$  \\ 
\cline{1-2}\cline{5-5}
\vrule width 0pt height 11pt
4 & $2, 4, 6$ &   &  & $\sigma(M) \neq 0$\\ 
\hline
\vrule width 0pt height 11pt
5 & $2k$ &  \multirow{2}{2em}{\hfill even\hfill}  & 
odd &  \\ 
\cline{1-2}\cline{4-5}
\vrule width 0pt height 11pt
6 & $k$ &   & {\hfill even\hfill}  & \\  
\hline
\end{tabular}

\vspace{2mm}
    \end{center}
In addition, if $d$ is attainable, then so is $h^2 d$ for all integers $h \geq 1$.
\end{proposition}

To prove this, we need the following preliminary lemmas. We start with the following direct consequence of \cite[Lemma 5.3]{PiergalliniZuddas1}. 
\begin{lemma}[Piergallini--Zuddas]\label{lem:pz}
  For every integer $d \geq 1$, there are isometric embeddings
  \[d \cdot H \hookrightarrow H \quad \text{and} \quad 2d \cdot H \hookrightarrow \langle 1 \rangle \oplus \langle -1 \rangle.\]
\end{lemma}
\begin{proof}
    This directly follows from \cite[Lemma 5.3]{PiergalliniZuddas1}. 
\end{proof}

\begin{lemma}\label{lem:hyperb}
For every even $d \geq 2$, there is an isometric embedding
\[\langle d\rangle \oplus \langle -d \rangle \hookrightarrow H.\]
\end{lemma}
\begin{proof}
    Let $e_1, e_2$ be a basis of the hyperbolic lattice $H$ such that $e_i^2=0$ and $e_1\cdot e_2=e_2 \cdot e_1=1$. For every $k\geq 1$, it is easy to verify that the vectors $e_1+ ke_2$ and $e_1-ke_2$ span a sublattice of $H$ isomorphic to $\langle 2k \rangle \oplus \langle -2k \rangle$, and let $2k = d$. 
\end{proof}

\begin{lemma}\label{lem:E82}
For every integer $d \geq 1$
there are isometric embeddings
\[d \cdot E_8 \hookrightarrow E_8 \quad \text{and} \quad d \cdot E_8 \hookrightarrow \oplus_8 H.\]
\end{lemma}

\begin{proof}
    We begin by constructing the embedding $d\cdot E_8 \hookrightarrow E_8$.
    
    Following the classical paper by Coxeter \cite{C}, we consider the sublattice $J$ of the octonion algebra $\mathbb{O}$, having as basis the columns of the following matrix
    \[\left(
\begin{array}{cccccccc}
 1 & 0 & 0 & 0 & 0 & -\frac{1}{2} & -\frac{1}{2} &
   -\frac{1}{2} \\[4pt]
 0 & 1 & 0 & 0 & \frac{1}{2} & 0 & \frac{1}{2} &
   -\frac{1}{2} \\[4pt]
 0 & 0 & 1 & 0 & \frac{1}{2} & -\frac{1}{2} & 0 &
   \frac{1}{2} \\[4pt]
 0 & 0 & 0 & 1 & \frac{1}{2} & \frac{1}{2} & -\frac{1}{2}
   & 0 \\[4pt]
 0 & 0 & 0 & 0 & \frac{1}{2} & 0 & 0 & 0 \\[4pt]
 0 & 0 & 0 & 0 & 0 & \frac{1}{2} & 0 & 0 \\[4pt]
 0 & 0 & 0 & 0 & 0 & 0 & \frac{1}{2} & 0 \\[4pt]
 0 & 0 & 0 & 0 & 0 & 0 & 0 & \frac{1}{2} \\
\end{array}
\right).\] This basis is precisely the one in Formula (5.1) in Coxeter's paper. Notice that the first four columns correspond, respectively, to the quaternion units $1, i, j, k\in \mathbb{H} \subset \mathbb{O}$. Moreover, the lattice $J$ is closed under octonion multiplication.

We endow $J$ with twice the standard scalar product of $\mathbb{O}$. This makes $J$ into an integral unimodular positive definite even lattice and, by the uniqueness theorem of Mordell \cite{Mor1938}, there is an isomorphism between $J$ and $E_8$. 

Given $d \geq 1$, by Lagrange's four squares theorem one can find $a_1,a_2,a_3,a_4\in \Z$ such that $d=a_1^2+a_2^2+a_3^2+a_4^2$. Then the octonion $\alpha=a_1+a_2i+a_3j+a_4k\in \mathbb{H}\subset \mathbb{O}$ is such that $|\alpha|^2=d$.

It follows that the left multiplication by $\alpha$ is a $|\alpha|$-dilation $J \rightarrow J$ which induces an isometric embedding $d \cdot E_8 \hookrightarrow E_8$ up to the above mentioned isomorphism $E_8\cong J$. Indeed, setting $\alpha' = \alpha/\sqrt d$, we have that
    \[\langle \alpha x,\alpha y \rangle = \langle \sqrt d \, \alpha' x, \sqrt d \, \alpha' y \rangle =d \langle \alpha 'x, \alpha' y \rangle=d \langle x,y \rangle\]
    for every $x, y \in \mathbb{O}$, since the multiplication by the unit octonion $\alpha'$ yields a unitary transformation of $\mathbb{O}$.

\smallskip
    We now construct an embedding $d\cdot E_8 \hookrightarrow \oplus_8 H$.
    There is an isomorphism $E_8 \oplus (-E_8) \cong \oplus_8 H$ because they are both even indefinite unimodular lattices with the same rank and signature. Then the desired embedding can be obtained as the composition \[d \cdot E_8 \hookrightarrow E_8 \hookrightarrow E_8\oplus (-E_8) \cong \oplus_8 H,\]
    where the second arrow is the canonical embedding in the first direct summand.
\end{proof}

\begin{lemma}\label{lem:emb}
For every even $d \geq 2$, there is an isometric embedding
\[d \cdot E_8 \hookrightarrow \oplus_8 \langle 1 \rangle.\]
\end{lemma}

\begin{proof} Consider the matrix
    \[L_2 = \left(\begin{matrix}
    1 & 2 & 1 & 0 & 0 & 0 & 0 & 0\, \\
    0 & 0 & 1 & 2 & 1 & 0 & 0 & 0\, \\
    0 & 0 & 0 & 0 & 1 & 2 & 1 & 0\, \\
    0 & 0 & 0 & 0 & 1 & 0 & 0 & 2\, \\
    0 & 0 & 1 & 0 & -1 & 0 & 1 & 0\, \\
    1 & 0 & -1 & 0 & 0 & 0 & 1 & 0\, \\
    1 & 0 & 0 & 0 & 0 & 0 & -1 & 0\, \\
    -1 & 0 & 0 & 0 & 0 & 0 & 0 & 0\,\, 
    \end{matrix}\,\right)\]
    and observe that ${^t\hspace{-1pt}}L_2 L_2 = 2 E_8$, which implies that the sublattice of $\oplus_8 \langle 1 \rangle$ spanned by the columns of $L_2$ is isomorphic to $2 \cdot E_8$.
    Notice that $L_2=2 G^{-1}$, where $G$ is the matrix considered in \cite[proof of Lemma 5.2]{PiergalliniZuddas1}.

    Let $d \geq 2$ be an even integer. We have a sequence of embeddings
    \[d \cdot E_8 =2 \cdot (d/2 \cdot E_8) \hookrightarrow 2 \cdot E_8 \hookrightarrow \oplus_8 \langle 1 \rangle,\]
    where the existence of the first arrow follow from Lemma \ref{lem:E82} and the second arrow is the one defined by the matrix $L_2$.
\end{proof}

\begin{proof}[Proof of Proposition \ref{prop:inequalities}]
First of all we observe that the last part of the statement is immediate, as it is sufficient to rescale by $h$ any given isometric embedding $\varphi\: d \cdot I_N \hookrightarrow I_M$.

Suppose that the inequalities $b_2^\pm(M)\geq b_2^\pm(N)$ hold. We subsequently analyze all the possible cases as in the table in the statement.

\medskip

\Case{1: $\beta_N$ and $\beta_M$ odd.} In this case $I_N$ and $I_M$ are diagonalizable by Donaldson's theorem \cite{Do87} and the Serre classification of indefinite integral unimodular inner product spaces \cite{Se73, MH73}. The inequalities $b_2^{\pm}(M) \geq b_2^{\pm}(N)$ then trivially imply the existence of an embedding $I_N \hookrightarrow I_M$.

\medskip

\Case{2: $\beta_N$ and $\beta_M$ odd, $b_2(N) \leq b_2(M)/2$.} There is an isometric embedding \[5 \cdot I_N \hookrightarrow I_M\] by \cite[Lemma 5.2]{PiergalliniZuddas1}.

\medskip 

\Case{3: $\beta_N$ odd, $\beta_M$ even and $\sigma(M)=0$.} In this case, $I_N$ is diagonalizable while $I_M \cong \oplus_{b^{\pm}_2(M)} H$ by the Serre classification. Lemma \ref{lem:hyperb} implies the existence of an embedding \[2k \cdot I_N \hookrightarrow I_M\] for every $k \geq 1$.

\medskip

\Case{4: $\beta_N$ odd, $\beta_M$ even and $\sigma(M)\neq 0$.} This case follows from an argument employed in \cite[Lemma 5.2]{PiergalliniZuddas1}. More precisely, there is an isomorphism 
\begin{equation}\label{I_M/eqn} I_M \cong \oplus_{\epsilon(M)} (\pm E_8) \oplus_{\eta(M)} H
\end{equation}
by Donaldson's theorem and the Serre classification, where we set $\epsilon(M)=|\sigma(M)|/8$ and $\eta(M)=b_2^{\mp}(M)$. In the proof of \cite[Lemma 5.2]{PiergalliniZuddas1}, it is shown that there are isometric embeddings \[\oplus_8 \langle d \rangle \hookrightarrow E_8\] for $d\in\{2,4,6\}$. Since we have also isometric embeddings \[\langle d \rangle \oplus \langle - d \rangle  \hookrightarrow H\] for the same values of $d$ by Lemma \ref{lem:hyperb}, the conclusion follows.
    
\medskip 

\Cases{5: $\beta_N$ even and $\beta_M$ odd.} 
By Donaldson's theorem and the Serre classification, there is an isomorphism 
\begin{equation}\label{I_N/eqn}
I_N\cong \oplus_{\epsilon(N)} (\pm E_8) \oplus_{\eta(N)} H,
\end{equation}
where $\epsilon(N)$ and $\eta(N)$ are as in case 4. Lemma \ref{lem:pz} and Lemma \ref{lem:emb} imply the existence of an isometric embedding
\[d \cdot \left(\oplus_{\epsilon(N)} (\pm E_8) \oplus_{\eta(N)} H \right)\hookrightarrow \oplus_{b_2^{+}(N)} \langle 1 \rangle \oplus_{b_2^{-}(N)} \langle -1 \rangle \hookrightarrow I_M\]
for every even $d\geq 2$.

\medskip 

\Case{6: $\beta_N$ and $\beta_M$ even.} 
We consider again the decompositions \eqref{I_M/eqn} and \eqref{I_N/eqn} with $\epsilon(M), \eta(M),\epsilon(N), \eta(N) \geq 0$.

Suppose first that $\sigma(N)\geq 0$ and $\sigma(M)\leq 0$. By assumption, we have
\[\eta(N) + 8\epsilon(N) = b_2^-(N) + \sigma(N) = b_2^+(N)\leq b_2^+(M)=\eta(M).\]
This means that the number of the hyperbolic direct summands plus eight times the number of the $E_8$ direct summands of $I_N$ is less than or equal to the number of the hyperbolic direct summands of $I_M$.

Therefore, inside $I_M$ there are sufficiently many hyperbolic direct summands in which we can embed $d \cdot I_N$ for every integer $d \geq 1$. This follows from Lemma \ref{lem:E82} for the $E_8$ direct summands and from \cite[Lemma 5.3]{PiergalliniZuddas1} which guarantees that an embedding $d\cdot H \hookrightarrow H$ exists for every $d\geq 1$. 

If $\sigma(N)\geq 0$ and $\sigma(M) \geq 0$, we have that 
\begin{equation}\label{eq:dis}
   b_2^-(N) \leq b_2^-(M) \quad \text{and} \quad \sigma(N)+b_2^-(N) \leq \sigma(M)+b_2^-(M). 
\end{equation}
The first inequality in \eqref{eq:dis} means that $\eta(N)\leq \eta(M)$. We can hence embed $d$ times the hyperbolic direct summands of $I_N$ into the hyperbolic direct summands of $I_M$ again by \cite[Lemma 5.3]{PiergalliniZuddas1}.

If $\sigma(N)\leq \sigma(M)$, we have that $\epsilon(N)\leq \epsilon(M)$ and we then use Lemma \ref{lem:E82} to embed $d$ times the $E_8$ direct summands of $I_N$ into the $E_8$ direct summands of $I_M$. 

If instead $\sigma(N) > \sigma(M)$, by the second inequality in \eqref{eq:dis} we have that
\[\sigma(N)-\sigma(M)\leq b_2^-(M)-b_2^-(N),\] hence
\[8(\epsilon(N)-\epsilon(M))\leq \eta(M)-\eta(N).\]
Then we can embed the direct summands $d\cdot E_8$ of $d\cdot I_N$ into $I_M$ by sending $\epsilon(M)$ of them into the $E_8$ direct summands of $I_M$, and the remaining $\epsilon(N) - \epsilon(M)$ into the hyperbolic part of $I_M$.

The case where $\sigma(N)\leq 0$ can be handled analogously, depending on the sign of $\sigma(M)$, and it is hence omitted.

We now prove the converse. Suppose that there is an isometric embedding of $d \cdot I_N$ inside $I_M$ for some positive integer $d$. Then the inequalities $b_2^{\pm}(N) \leq b_2^{\pm}(M)$ follow from the very definition of the $b_2^{\pm}$ invariants of a closed oriented 4-manifold.
\end{proof}

\section{Fillability of 3-dimensional branched coverings}

\begin{definition}\label{def:cob}
 The cobordism group of simple $d$-fold branched coverings over $S^3$ is defined as
 \[\Gamma_d(S^3)=\{p \colon Y \darrow{d} S^3 \text{ simple }\mid B_p \subset S^3 \text{ a link} \}/{\sim},\]
 where $p_1 \sim p_2$ if and only if there exists a simple $d$-fold branched covering \[c: C \darrow{d} S^3 \times [0,1]\] which is a cobordism between $p_1$ and $p_2$. Moreover, we require that the branch set $B_c\subset S^3 \times [0,1]$ is a proper locally flat PL surface. 
 
 We endow $\Gamma_d(S^3)$ with a binary operation by setting $[p_1]+[p_2]=[p]$, where $p$ is the simple $d$-fold covering over $S^3$ whose labeled branch set $B_p$ is given by the disjoint separate union of $B_{p_1}$ and $B_{p_2}$ inside $S^3$, for all $[p_1],[p_2]\in \Gamma_d(S^3)$. This makes $\Gamma_d(S^3)$ an abelian group.
 \end{definition}

 The following result will be our main tool to extend a given simple branched covering over a $4$-handle.

\begin{theorem}\label{thm:fillable}
Let $W^4$ be a compact connected PL $4$-manifold with non-empty boundary, and let \[p \colon \partial W^4 \darrow{d} S^3\] be a simple $d$-fold covering branched over a link $B_p \subset S^3$ such that $[p]=0\in \Gamma_d(S^3)$, where $d \geq 4$. Then $p$ extends to a simple $d$-fold covering \[ q \colon W^4 \darrow{d} B^4\] branched over a PL embedded surface $B_q \subset B^4$ which is locally flat if $d \geq 5$, while it has at most nodal singularities if $d=4$.     
\end{theorem}
 
 \begin{proof}
     It is a direct consequence of \cite[Theorem 1.4]{PiergalliniZuddas2}, in light of \cite[Lemma 1.14]{BPZ}. More precisely, the condition $[p]=0\in \Gamma_d(S^3)$ is equivalent to asking that the branch set $B_p\subset S^3$ bounds a locally flat proper surface in $B^4$ over which the monodromy representation of $p$ extends. But then one can choose such surface to be ribbon because of \cite[Lemma 1.14]{BPZ} and the conclusion is a direct application of \cite[Theorem 1.4]{PiergalliniZuddas2}.
 \end{proof}

\section{Proof of Theorem \ref{thmA}}
\label{proofs/sec}
 This section is devoted to the proof of Theorem \ref{thmA}. We first state and prove the following lemma.

\begin{lemma}\label{lem:bands}
    Let $A= A_1 \sqcup \dots \sqcup A_n$ and $B=B_1 \sqcup \dots \sqcup B_n$ be two oriented links in $S^3$ such that $\lk(A_i,A_j)=\lk(B_i,B_j)$ for every $i \neq j$. Then, for every $i=1, \dots, n$, $A_i$ is related to $B_i$ via a finite sequence of oriented band attachments and ambient isotopies. Moreover, each band can be chosen to intersect the link only at its attaching edges, both lying in the same connected component of the link.
\end{lemma}

\begin{proof}
For every $i=1,\dots, n$, let $F_i\subset S^3\times [0,1]$ be a PL locally flat proper oriented surface such that $\partial F_i=A_i \sqcup \widebar{B}_i$, where we view $A\subset S^3\times \{0\}$ and $B \subset S^3 \times \{1\}$ with $S^3\times\{0\}$ and $S^3 \times \{1\}$ oriented as the boundary of $S^3 \times [0,1]$. Without loss of generality, we can assume that each $F_i$ is connected and that $F=F_1 \cup \dots \cup F_n \subset S^3 \times [0,1]$ is self-transversally immersed.
Now, notice that the algebraic intersection between $F_i$ and $F_j$, for $i \neq j$, is given by \[F_i\cdot F_j=\lk(A_i,A_j)-\lk(\widebar{B}_i,\widebar{B}_j)=0,\] where the minus sign arises from the fact that $S^3\times \{0\}$ and $S^3 \times \{1\}$ are endowed with opposite orientations. Since $F_i \cdot F_j$ vanishes, we can geometrically remove all the intersection points by tubing them in pairs with opposite signs. We can hence reduce to the case in which $F$ is nonsingular. Without loss of generality, we can assume that the restriction to $F_i$ of the height function $S^3\times [0,1] \to [0,1]$ is Morse. If all the critical points of this map are saddle points, then the conclusion follows. Suppose now that such map has a local minimum in $\Int F_i$. We can then push this minimum to the lower boundary component $S^3 \times \{0\}$ along an arc that does not intersect $F$ elsewhere. This procedure introduces a new trivial component in the link $A$, which can be trivially merged to $A_i$ with a suitable oriented band in $S^3 \times \{0\}$. By pushing the interior of such band into $\text{Int } S^3 \times [0,1]$, we get a link in $S^3 \times \{0\}$ which is ambiently isotopic to $A$. We repeat this procedure for all the local minima. In a similar way, we can remove the local maxima, and the conclusion follows.
\end{proof}

\begin{proof}[Proof of Theorem \ref{thmA}]
If $p \colon M \darrow{d} N$ is a simple branched covering, then Lemma \ref{lem:necessary} implies the existence of an isometric embedding of lattices $d \cdot I_N \hookrightarrow I_M$. 

Let us now prove the converse. Let \[L = L_1\cup\cdots \cup L_n \subset S^3\] be an oriented framed link representing the $2$-handlebody $N_2$ of $N$. Then $L$ induces a basis $\alpha_1,\dots, \alpha_n$ of $H_2(N; \Z)$, where each $\alpha_i$ is the homology class of a closed connected oriented surface in $N$ obtained by taking the union of a Seifert surface of $L_i$ pushed inside $B^4$ and of the core of the 2-handle attached along $L_i$. Let $A = (a_{ij} = \beta_N(\alpha_i, \alpha_j)) \in \GL(n, \Z)$ be the matrix representing the intersection lattice $I_N$ of $N$ with respect to the basis $\alpha_1,\dots, \alpha_n$, so that $a_{ii}$ is the framing of $L_i$ for $i=1, \dots, n$ and $a_{ij} = \lk(L_i, L_j)$ for all $i\neq j$. 

Suppose now that there is an isometric embedding of lattices \[\varphi \colon d\cdot I_N \hookrightarrow I_M.\] 
Let $S_i\subset M$ be a closed connected oriented locally flat PL surface representing $\varphi(\alpha_i)$ for $i=1,\dots, n$. In particular, we have that
\[\beta_M([S_i],[S_j])= \beta_M(\varphi(\alpha_i), \varphi(\alpha_j))=d \cdot \beta_N(\alpha_i, \alpha_j)=d \cdot a_{ij}\]
for every $i,j\in\{1, \dots, n\}$. Up to possibly a small perturbation, we can assume that the surfaces $S_i$ are in general position, so that they are transversal to each other and no three of them have non-empty intersection. Up to tubing, we can also assume that the geometric intersection of $S_i$ and $S_j$ equals the algebraic intersection for all $i\neq j$. A regular neighborhood of \[S = S_1\cup\cdots\cup S_n\] is then a plumbing of normal disk bundles over $S_1, \dots, S_n$ with intersection matrix $d \cdot A$. At this point, pick a small $4$-ball around each double point of $S$. If there is a surface $S_i$ such that $S_i \cap S_j =\emptyset $ for every $j \neq i$, pick a $4$-ball that intersects $S_i$ along a properly embedded trivial $2$-disk. Now join all these balls by suitable $1$-handles embedded in the complement of a regular neighborhood of $S$ to get a $4$-ball $B^4 \subset M$ whose boundary $S^3$ is transversal to all the surfaces $S_i$ and let $\widetilde L_i = S_i \cap S^3$ for $i=1, \dots, n$. In this way we get an oriented link \[\widetilde L \defeq \widetilde L_1 \sqcup \dots \sqcup \widetilde L_n \subset S^3.
\] where $\widetilde{L}_i$ is oriented as the boundary of $S_i \cap B^4$.

Now notice that, if $\beta \subset S^3$ is an oriented band attached to $\widetilde L_i$ in the complement of $\cup_{j\neq i} \widetilde L_j$ which realizes a homology of the knot $\widetilde L_i$ inside $S^3 - \cup_{j\neq i} \widetilde L_j$, we can push $\beta$ inside and outside $B^4$ to get bands $\beta'$ and $\beta''$ respectively. Their union $\beta'\cup\beta''$ is then an oriented tube that realizes a homology of $S_i$ in $M$. By Lemma \ref{lem:bands}, we can hence modify each $S_i$ up to oriented homology so that the intersection of the plumbing $S$ with $S^3=\partial B^4$ is any oriented link with linking matrix $d\cdot A$. Notice that these tubing operations do not change the intersection numbers $\beta_M([S_i] , [S_j])= d\cdot a_{ij}$ for every $i,j$, since they change the surfaces without affecting their homology classes.

We are now ready to construct the desired branched covering between $M$ and $N$. The starting point is the $d$-fold simple covering \[q_0 \colon B^4 \darrow{d} B^4\] branched over $d-1$ pairwise separated trivial properly embedded $2$-disks with labeling $\trans{1}{2}, \dots, \trans{d-1}{d}$ respectively. Up to a suitable isotopy of the branch set $B_{q_0}$, we can assume that $L$ is disjoint from $B_{q_0}$ and each component $L_i$ of $L$ has a cycle of length $d$ as monodromy, so that its preimage $q_0^{-1}(L_i) \subset S^3$ is connected, see Figure \ref{figura1}.

\begin{figure}[ht]
\includegraphics[scale=0.8]{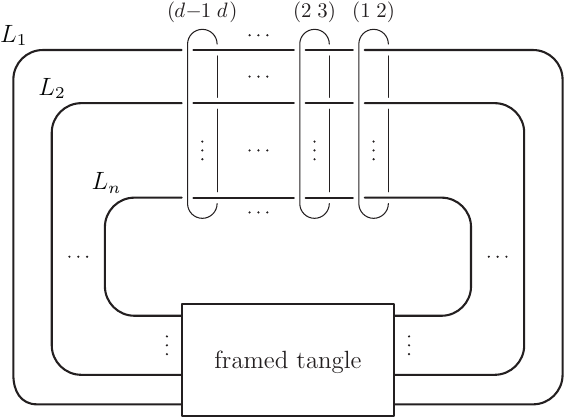}
\caption{Labeled branch set of $q_0|_{\partial}$.}\label{figura1}
\end{figure}

Now notice that, by construction, the intersection matrix of the link \[q_0^{-1}(L_1) \sqcup \dots \sqcup q_0^{-1}(L_n)\] is equal to $d \cdot A$ and this implies that each component $q_0^{-1}(L_i)$ is realizable as the intersection 
$\widetilde{L}_i$ of $S_i$ with $S^3$ as above. 

We are now going to extend $q_0$ over the $2$-handles of $N$. For every $i=1, \dots, n$, \[\Sigma_i \defeq S_i \setminus \text{Int  } B^4\] is a surface with one boundary component $\widetilde{L}_i=S_i \cap S^3$ with a $d$-fold covering \[p_i \colon \Sigma_i \darrow{d} B^2\] branched over $2g(\Sigma_i)+d-1$ points, the meridians of $2g(\Sigma_i)+1$ of them have monodromy $\trans{1}{2}$ and there is exactly one branch point corresponding to each transposition $\trans{2}{3}, \dots, \trans{d-1}{d}$. Here $g(\Sigma_i)$ denotes the genus of the surface $\Sigma_i$. We remark that such $p_i$ can be constructed by stabilizing the $2$-fold branched covering $\Sigma_i \darrow{2} B^2$ obtained by removing two equivariant disks from the domain and target of the quotient map $\Sigma_i \cup_{\partial} B^2 \darrow{2} S^2$ of the hyperelliptic involution on the genus $g(\Sigma_i)$ closed oriented surface $\Sigma_i \cup_{\partial} B^2$. Now notice that $\nu(S_i)\setminus \text{Int } B^4 \cong \Sigma_i \times B^2$ meets $B^4$ along $\nu(\widetilde{L}_i) \subset S^3= \partial B^4$. In particular, $\partial \Sigma_i \times B^2$ is identified with $\nu(\widetilde{L}_i)$ with framing $\beta_M([S_i],[S_i])=d \cdot a_{ii}$. Since this coincides with the pull-back of the framing $a_{ii}$ of $L_i$ under $q_0$, we can extend $q_0 \colon B^4 \darrow{d} B^4$ over the $2$-handles of $N$ by the simple branched coverings \[\nu(S_i) \setminus \text{Int } B^4 \cong \Sigma_i \times B^2 \darrow{d} B^2 \times B^2\] given by $p_i \times \text{id}_{B^2}$. In particular, extending $q_0$ over the $2$-handle attached along $L_i$ adds to the branch set $2g(\Sigma_i)+d-1$ embedded $2$-disks given by parallel copies of the co-core, where $2g(\Sigma_i)+1$ of them have monodromy $\trans{1}{2}$ and the other $2$-disks are labeled by $\trans{2}{3}, \dots, \trans{d-1}{d}$ respectively. In this way, we have just found a simple $d$-fold branched covering
\[q_2 \colon B^4 \cup \nu(S) \darrow{d} N_2\]
where $N_2$ denotes the $2$-handlebody of $N$, as above.

In order to conclude, we need to further extend $q_2$ over the unique $4$-handle of $N$. We will show that the restriction of $q_2$ to the boundary is trivial in the cobordism group $\Gamma_d(S^3)$, see Definition \ref{def:cob}. The conclusion will then follow from Theorem \ref{thm:fillable}.

One can visualize the branch set of $q_2|_{\partial} \colon \partial \nu(S) \darrow{d} S^3$ as follows. First of all, recall that we isotoped the link $L$ as in Figure \ref{figura1}, so that the meridian of each component $L_i$ has a permutation of maximum length as a monodromy. Notice that, since $N$ has no $3$-handles, $S^3$ is diffeomorphic to the boundary of the $2$-handlebody of $N$ and the framed link $L$ provides a Dehn surgery diagram for $S^3$. In this surgery diagram, one can visualize the dual $2$-handles of $N$ as $0$-framed meridians of the components of $L$, see \cite[Example 5.5.5]{GompfStipsicz}. In particular, in the surgery presentation of $S^3$ given by the framed link $L$, the labeled branch set of $q_2|_{\partial}$ is the union of $d-1$ components as in Figure \ref{figura1} and $2g(\Sigma_i)+d-1$ parallel meridians of each $L_i$, $2g(\Sigma_i)+1$ of them with monodromy $\trans{1}{2}$ and exactly one of them for each monodromy of type $\trans{2}{3}, \dots, \trans{d-1}{d}$, see Figure \ref{figura2}.

\begin{figure}[ht]
\includegraphics[scale=0.8]{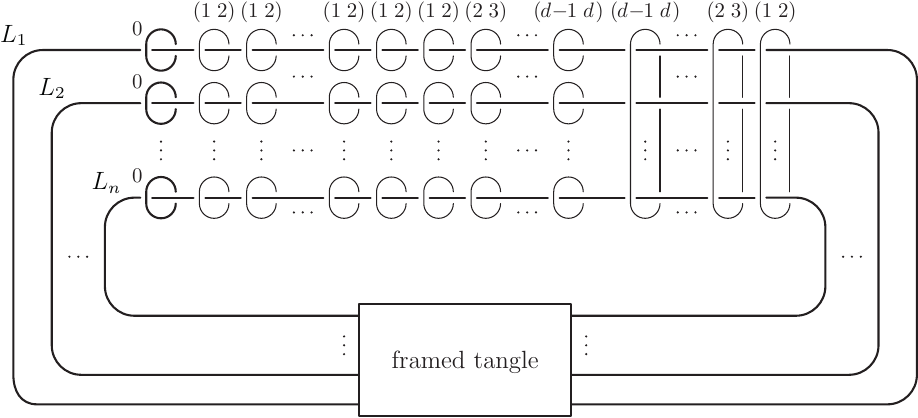}
\caption{Labeled branch set of $q_2|_{\partial}$ in the surgery diagram of $S^3$ given by the framed link $L$.}\label{figura2}
\end{figure}

\begin{figure}[htb]
\includegraphics[scale=0.8]{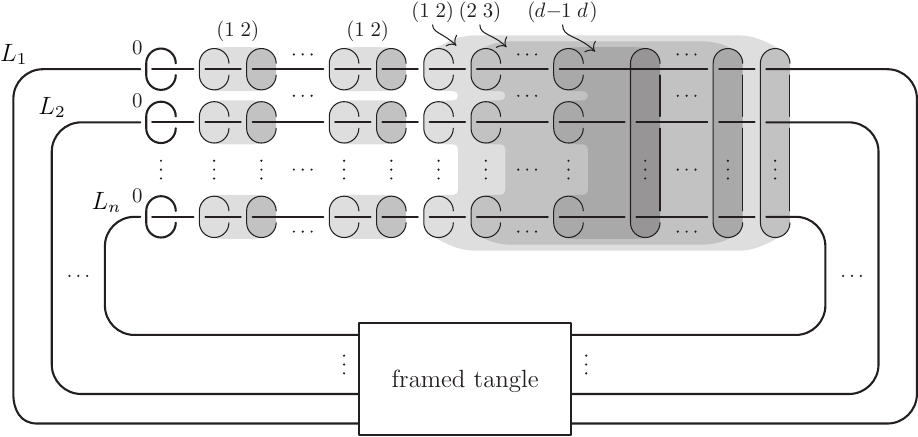}
\caption{Labeled surface bounding $B_{q_2|_{\partial}}$.}\label{figura3}
\end{figure}

We will now construct a surface bounding $B_{q_2|_{\partial}}$, over which the monodromy $\omega_{q_2|_{\partial}}$ extends. This will imply that $[q_2|_{\partial}]=0\in \Gamma_d(S^3)$ and the conclusion will follow from Theorem \ref{thm:fillable}. For every $i=1, \dots, n$, the first $2g(\Sigma_i)$ meridians of $L_i$ with monodromy $\trans{1}{2}$ bound $g(\Sigma_i)$ disjointly embedded labeled cylinders. The union of the remaining $\trans{1}{2}$-meridians of $L$ with the $\trans{1}{2}$-component of $B_{q_2|_{\partial}}$ coming from $B_{q_0}$ bounds an embedded colored surface, see Figure \ref{figura3}. In particular, such surface is an embedded $S^2$ with $n+1$ punctures. Similarly, for $i=2, \dots, d-1$, there are disjointly embedded genus $0$ labeled surfaces with $n+1$ boundary components, co-bounding the $\trans{i}{i+1}$-framed meridians of $L$ and the $\trans{i}{i+1}$-framed component coming from the branch set of $q_0$. The conclusion follows.
\end{proof}

\section{Proofs of Corollaries \ref{cor:b} and \ref{cor:c}}\label{proofcorBC}

\begin{proof}[Proof of Corollary \ref{cor:b}]
One implication is trivial, since any branched covering $p \colon M \rightarrow N$ is a dominant map. For the converse, assume that $f \colon M \rightarrow N$ is dominant and that $\deg f \geq 1$. By \cite{Hopf} (see also \cite[Proposition 2.6]{Harpe}), the pull-back map
\[f^* \colon H^*(N;\mathbb{Q}) \longrightarrow H^*(M;\mathbb{Q}) \] is an injective morphism of $\mathbb{Q}$-algebras. This implies the existence of an isometric embedding $\deg f \cdot I_N \hookrightarrow I_M$. The conclusion then follows from Theorem \ref{thmA}.
 \end{proof}
 \begin{proof}[Proof of Corollary \ref{cor:c}]
 The $\K$ surface satisfies the hypothesis of Theorem \ref{thmA}, since it admits a handle decomposition without $1$- and $3$-handles (see \cite[Figure 8.15]{GompfStipsicz} for $n=2$ or \cite[Figure 2.15]{HKK}). Therefore, there exists a $d$-fold branched covering $M \darrow{d} \K$ if and only if there is an isometric embedding of lattices $d \cdot I_{\K} \hookrightarrow I_M$. Given that $\beta_{\K}$ is even, $b_2^+(\K) = 3$ and $b_2^-(\K) = 19$, the conclusion follows from Proposition \ref{prop:inequalities}.
 \end{proof}

\end{document}